\documentclass[11 pt]{amsart}
\usepackage{amscd,amsfonts,amssymb,amsmath}
\usepackage{hyperref}
\usepackage{epsfig}
\newtheorem{theorem}{Theorem}[section]
\newtheorem{corollary}[theorem]{Corollary}

\newtheorem{proposition}[theorem]{Proposition}
\theoremstyle{definition}

\newtheorem{example}[theorem]{Example}
\newtheorem{remark}[theorem]{Remark}
\numberwithin{equation}{subsection}

\usepackage[all,cmtip]{xy}
\usepackage{xcolor}
\newcommand{\Alex}{\operatorname{Alex}}
\newcommand{\Aut}{\operatorname{Aut}}

\newcommand{\Conj}{\operatorname{Conj}}
\newcommand{\Core}{\operatorname{Core}}

\newcommand{\T}{\operatorname{T}}

\newcommand{\R}{\operatorname{R}}

\begin{document}

\title{ Automorphisms and Antiautomorphisms  of quandles}
\author{Birama Sangare}

\date{\today}
\address{Sobolev Institute of Mathematics and Novosibirsk State University, Novosibirsk 630090, Russia.}

\email{b.sangare@g.nsu.ru}

\maketitle

\begin{abstract}
In this paper we provide the  conditions under which an automorphism or an antiautomorphism of a group $G$ induces an automorphism or an antiautomorphism of the $m$-conjugation quandle
$\operatorname{Conj_{m}}(G),\,\, m\in \mathbb{Z} $, the core quandle $\operatorname{Core}(G)$, the generalized Alexander quandle $\operatorname{Alex}(G,\phi)$ where $\phi\in \operatorname{Aut}(G)$ and some others. We also construct automorphisms  of these quandles that do not originate from $G$.
\end{abstract}
\begin{center}
    \section*{Introduction}
\end{center}

The concept of quandles was first presented in the independent works of Joyce \cite{Joyce} and Matveev \cite{Matveev} for the classification of knots in the $3$-sphere. Currently, quandles are being studied from algebraic points of view (see, for example, \cite{VMM, Y, Andruskiewitsch}). 
One of the problems in quandle theory is the description of quandle automorphisms and antiautomorphisms. The description of quandle automorphisms  has been discussed  in  papers \cite{BS, BNS, KSS, MS, EMR, BDS}. Recent developments in quandle theory have shown that antiautomorphisms can be used as building blocks for more complex structures. For instance, quandles can be constructed over non-commutative groups using arbitrary antiautomorphisms and they play a vital role in the theory of quandle derivations and extensions (see, for example \cite{BS, MN}). Understanding these maps is therefore essential for the broader project of classifying all possible quandle structures and their representations. Since many functors from the category of groups to the category of quandles are defined explicitly via group operations, the quandle  effectively acts as a specialized projection of a group. For examples: $\operatorname{Conj}_{m}$, $\operatorname{Core}$ and $\operatorname{Alex}$ (see §\ref{sec-prelim} for the definitions). Consequently, a deep algebraic motivation exists for establishing the precise conditions under which an automorphism or an antiautomorphism of group induces  an automorphism or an  antiautomorphism of a corresponding quandle. In this paper, we study antiautomorphism of quandle by analogy  to antiautomorphism of group. We can formulate the following general problem:

{\bf Problem.} Let $\mathcal{Q}$ be a functor from the category of groups to the category of quandles, and let $G$ be a group. What are the connections between the automorphism groups $\operatorname{Aut}(G)$ and $\operatorname{Aut}(\mathcal{Q}(G))$? In particular, for which functors $\mathcal{Q}$ and groups $G$ does the equality $\operatorname{Aut}(G)=\operatorname{Aut}(\mathcal{Q}(G))$ hold?

Since $\mathcal{Q}$ is a functor, then it induces a homomorphism $\operatorname{Aut}(G)\to \operatorname{Aut}(\mathcal{Q}(G))$ for any group $G$.

In \cite{BNS} it was proved that $\operatorname{Aut(Conj}(G))=\operatorname{Aut}(G)$ if and only if the center $Z(G)$ of $G$ is trivial. For the functor Core, it is known (see \cite{BNS}) that $\operatorname{Aut(Core}(G))$ contains as a subgroup $Z(G)\rtimes\operatorname{Aut}(G)$. For the functor Alex, it was proved in \cite{BDS} that $\operatorname{Alex}(G,\phi)$ contains as a subgroup $Z(G)\rtimes C_{\operatorname{Aut}(G)}(\phi)$,  where $C_{\operatorname{Aut}(G)}(\phi)$ is the centralizer of $\phi$ in $\operatorname{Aut}(G)$, i.e. $$C_{\operatorname{Aut}(G)}(\phi)=\{\varphi\in\operatorname{Aut}(G)\,\,|\,\,\varphi\phi=\phi\varphi \}.$$ These two groups are equal if and only if $G$ is a finite abelian group and $\phi$ is a fixed-point-free automorphism of $G$.

 In this paper, we study problem above for automorphism groups $\operatorname{Aut}(G)$ and $\operatorname{Aut}(\mathcal{Q}(G))$, but also for the sets of antiautomorphisms $\operatorname{AAut}(G)$ and $ \operatorname{AAut}(\mathcal{Q}(G))$, in a group $G$ and  quandle $\mathcal{Q}(G)$ respectively. It is known (see \cite{SNB}) that if $\varphi\in\operatorname{Aut}(G)$, then the composition $\varphi\varepsilon$, where $\varepsilon: G\to G,\,\,\varepsilon(g)=g^{-1},\,g\in G$, is an antiautomorphism and any antiautomorphism of $G$ can be constructed in this manner. Therefore, for the groups, the description of automorphisms is the same as the description of antiautomorphisms. The situation is different for quandles. In particular, we prove that there are quandles with a non-trivial group of automorphisms that do not have antiautomorphisms.
 
 In recent years, new constructions of quandles from groups were introduced. In \cite{SNB},  some analogous of generalized Alexander quandle were defined. Verbal quandles were introduced in \cite{VMM}. For example, $\operatorname{Conj}_{m}(G)$ and $\operatorname{Core}(G)$ are types of verbal quandles. Also,  \cite{VMM} describes all verbal quandles that can be defined for an arbitrary group. In \cite{MN}, this result was generalized and all one parameter verbal quandles were described.

 This paper is organized as follows. In Section \ref{sec-prelim}, we introduce the fundamental definitions and notations .
 
 In Section \ref{s2}, we introduce the set of maps $$F=\{f_{a,b}: G\to G\,\,|\,\,a,b\in G,f_{a,b}(x)=axb, x\in G \}$$  and show that it forms a group under multiplication $f_{a,b}f_{c,d}=f_{ac,db},\,a,b,c,d\in G$. The set $H$ consisting of all multiplications by central elements i.e. $$H=\{f_{a}: G\to G\,\,|\,\,a\in Z(G),f_{a}(x)=ax, x\in G \}$$ is a normal subgroup of $F$. We will prove that $H\rtimes \operatorname{Aut}(G)\leq \operatorname{Aut(Conj_m}(G)),\,\,m\in\mathbb{Z}$ (Proposition \ref{T1}). This result generalizes \cite[Proposition $4$.$7$]{BDS}. This section ends with the Proposition \ref{proposition}, in which we prove that $G^{op} \rtimes C_{\operatorname{Aut}(G)}(\phi)\leq \operatorname{Aut(Alex} (G, \phi))$. This result also generalizes \cite[Proposition $4$.$1$]{BDS}. Here, $G^{op}$ is the group that is the opposite of $G$ i.e. $G^{op}=G$ as the set and the operation $\circ$ is defined by $a\circ b=ba$ for $a,b\in G$.
 
 In  Section \ref{s3}, we prove that if the center $Z(G)$ is trivial, then $(G\times G^{op})\rtimes \operatorname{Out}(G)\leq \operatorname{Aut(Core}(G))$ (Proposition \ref{proposition 2.10}). For the dihedral quandle $\operatorname{R}_{n}$ of order $n\neq 3$, we  prove that the set of antiautomorphisms is empty (Theorem \ref{Theorem 3.4}). On the other hand side it was proved in \cite[Theorem $2$.$1$]{EMR} that $\operatorname{Aut(R}_{n})=\mathbb{Z}_{n}\rtimes\mathbb{Z}_{n}^{\times}$, where $\mathbb{Z}_{n}^{\times}$ is the group of units of $\mathbb{Z}_{n}$.\\
 In Section \ref{s4}, we consider quandles $Q_{i}:=(G, \ast_{i}),\,i=1,2,3,4$, which were introduced in \cite{SNB}. We prove that $ C_{\operatorname{Aut}(G)}(\phi)\leq \operatorname{Aut}(Q_{i})$ (Theorem \ref{theorem 6}). We also prove that  $\varphi\in C_{\operatorname{AAut}(G)}(\phi)$ lies in $\operatorname{Aut}(Q_{i})$ if and only if $\phi$ is a central automorphism (Theorem \ref{theorem 3.2}), where $$C_{\operatorname{AAut}(G)}(\phi)=\{\psi\in\operatorname{AAut}(G)\,\,|\,\,\psi\phi=\phi\psi \}$$ is the set of antiautomorphisms of $G$ that commute with $\phi$. Recall that a central automorphism of a group $G$ is an automorphism $\theta$ such that $x^{-1}\theta(x)\in Z(G)$ for all $x\in G$.\\
 In  Section \ref{s5}, we consider verbal quandles with one parameter $c\in G$. For these quandles, $\,\,P_{i}$ is defined as $(G,\ast_{i}),\,i=1,2,3,4$, which were introduced in \cite{MN}. We prove that for a given $\varphi\in \operatorname{Aut}(G)$,  $\varphi$ induces an automorphism of $P_{i}$ if and only if $c^{-1}\varphi^{-1}(c)$ belongs to $Z(G)$ (Theorem \ref{theorem 4.1}). We also prove that if $\varphi\in \operatorname{Aut}(G)$ and $c\in\operatorname{Fix}(\varphi)$, then $\varphi$ induces an antiautomorphism of $P_{i}$ if and only if $x=[x, c^{-1}],\,\,x\in G$ (Theorem \ref{theorem 4.2}), where $\operatorname{Fix}(\varphi)=\{x\in G\,\,|\,\,\varphi(x)=x\}$ and $[a,b]=a^{-1}b^{-1}ab,\,a,b\in G$. We conclude the paper with a result that proves,  if $\psi\in \operatorname{AAut}(G)$ and $c\in\operatorname{Fix}(\psi)$, then $\psi$ induces an automorphism of $P_{i}$ if and only if $[c^{2},x^{-1}]=1,\,\,x\in G$ (Theorem \ref{theorem 4.2}).

\bigskip

\section{Definitions and Notations}\label{sec-prelim}
A {\it quandle} is a non-empty set $Q$ with a binary operation $(x,y) \mapsto x * y$ satisfying the following axioms:
\begin{enumerate}
\item[(Q1)] $x*x=x$ for all $x \in Q$;
\item[(Q2)]
for any $x,y \in Q$ there exists a unique $z \in Q$ such that $x=z*y$;
\item[(Q3)] $(x*y)*z=(x*z) * (y*z)$ for all $x,y,z \in Q$.
\end{enumerate}

An algebraic system satisfying only $(Q2)$ and $(Q3)$  is called a {\it rack}. Many interesting examples of quandles come from groups.
\par

\begin{itemize}
\item A quandle  $Q$ is called {\it trivial} if $x*y=x$ for all $x, y \in Q$.  Unlike groups, a trivial quandle can have arbitrary number of elements. We denote the $n$-element trivial quandle by $\T_n$.
 \item Let $m \in \mathbb{Z}$  and $G$ be a group, then the binary operation $x*y=x^{-m}xy^{m}$ turns $G$ into the quandle $\operatorname{Conj_{m}}(G)$ called the {\it $m$-conjugation quandle} of $G$. In particular, for $m=1$ we have $\operatorname{Conj}(G)$ called the {\it conjugation quandle}. If $G$ is abelian then the quandle $\operatorname{Conj_{m}}(G)$ is trivial.
\item A group $G$ with the binary operation $x*y= yx^{-1} y$ turns the set $G$ into the quandle $\Core(G)$ called the {\it core quandle} of $G$.  In particular, if $G= \mathbb{Z}_n$, the cyclic group of order $n$, then it is called the {\it dihedral quandle} and denoted by $\R_n$.
\item Let $G$ be a group and $\phi \in \Aut(G)$. Then the set $G$ with binary operation $x * y = \phi(xy^{-1})y$ forms a quandle $\Alex(G,\phi)$ called as the  {\it generalized Alexander quandle} of $G$ with respect to $\phi$.
    
\end{itemize}
\medskip
A quandle $Q$ is called {\it commutative} if $x * y = y * x$ for all $x, y \in Q$. A quandle $Q$ is called {\it cocommutative} if $x *^{-1} y = y * ^{-1}x$ for all $x, y \in Q$.\\
A bijection $\varphi : Q \to Q$ is called an automorphism of a quandle $Q$ if for every elements $x,y \in Q$
the equality $\varphi(x\ast y)=\varphi(x)\ast \varphi(y)$ hold. An antiautomorphism $\psi$ of $Q$ is a bijection  $Q$ that reverses the order of operation of $Q$,  i.e., $\psi(x\ast y)=\psi(y)\ast\psi(x)$ for all $x,y\in Q$. 
\begin{remark}
    In \cite{EH} and \cite{LuT} a bijection $\psi:Q \to Q$ is called an antiautomorphism if $\psi(x\ast^{-1}y)=\psi(x)\ast^{-1}\psi(y)$ for all $x,y\in Q$. This defintion is not equivalent to our definition. Indeed, if $Q$ is an involutory quandle, i.e., $\ast=\ast^{-1}$, then from identity $\psi(x\ast y)=\psi(x)\ast^{-1}\psi(y)$ follows that $\psi$ is an automorphism. In our case from identity $\psi(x\ast y)=\psi(y)\ast \psi(x)$ follows that  $\psi$ is an automorphism if $Q$ is a commutative quandle.\\
  Luc Ta proved in his blog \cite{LT} that a commutative quandle is cocommutative if and only if it is an involution quandle. Therefore, any antiautomorphism of a both commutative quandle and cocommutative quandle in this paper (resp. in the papers \cite{EH} and \cite{LuT}) is also automorphism.
  \end{remark}
\medskip

 For a group $G$, an automorphism $\phi \in\mathrm{\Aut }(G)$ and an antiautomorphism $\psi \in \mathrm{AAut}(G)$, in \cite{BNS} were defined 4 types of quandles:
\begin{itemize}
    \item [-] $Q_{1}:=(G, \ast_{1}): \,\,x\ast_{1} y= y\phi(y^{-1}x),\,\, x,y \in G$ and $\phi\in \operatorname{Aut}(G)$.
    \item [-]$Q_{2}:=(G, \ast_{2}): \,\,x\ast_{2} y= y\psi(yx^{-1}),\,\, x,y \in G$ and $\psi\in \operatorname{AAut}(G)$.
    \item [-] $Q_{3}:=(G, \ast_{3}): \,\,x\ast_{3} y=\psi(xy^{-1})y,\,\, x,y \in G,\,\,\, \psi\in \operatorname{AAut}(G)$ and
    
    $x\psi(y)x^{-1}=\psi(xyx^{-1})$.
    \item [-] $Q_{4}:=(G, \ast_{4}): \,\,x\ast_{4} y= y\psi(y^{-1}x),\,\, x,y \in G,\,\,\, \psi\in \operatorname{AAut}(G)$ and 

    $x\psi(y)x^{-1}=\psi(xyx^{-1})$.
\end{itemize}
In \cite{MN}, were defined one parameter verbal quandles. Let $c$ be a fixed element of $G$, then on the set $G$ it is possible to define 4 types of quandles:
\begin{itemize}
    \item [-] ${P}_{1}:=(G,\circ_{1}): x\circ_{1}y= yc^{-1}y^{-1}xc,\,\, x,y\in G$,
    \item [-]${P}_{2}:=(G,\circ_{2}): x\circ_{2}y=yc^{-1}xy^{-1}c,\,\, x,y\in G,$
    \item [-] ${P}_{3}:=(G,\circ_{3}): x\circ_{3}y=c^{-1}y^{-1}xcy,\,\, x,y\in G,$
    \item [-]${P}_{4}:=(G,\circ_{4}): x\circ_{4}y=c^{-1}xy^{-1}cy,\,\, x,y\in G.$ 
\end{itemize}

\par

%%%%%%%%%%%%%%%%%%%%%%%

%%%%%%%%%%%%%%%%%%%%%%%%%%%%%%%%%%%%%%%%%%%%%%%%

\section{ Automorphisms and antiautomorphisms on  $\operatorname{Conj_{m}}(G)$ and  $\operatorname{Alex}(G,\phi)$ }
\label{s2}

In this section, we determine  the conditions under which an automorphism or an antiautomorphism  of a group $G$ induces an automorphism or an antiautomorphism on the $m$-conjugation quandle $\operatorname{Conj_{m}}(G)$ and the Alexander quandle $\operatorname{Alex}(G,\phi)$.   

First, we prove the following proposition, which generalizes the result in \cite[Proposition $4$.$7$]{BDS}. To prove it, we introduce the set $H$ of all multiplications by central elements i.e. $$H=\{f_{a}: G\to G\,\,|\,\,a\in Z(G),f_{a}(x)=ax, x\in G \}.$$ We then consider the semi-direct product $H\rtimes \operatorname{Aut}(G)$, in which multiplication is defined by the rule $$(f_{a},\varphi_{1})(f_{b},\varphi_{2})=(f_{a}.\varphi_{1}f_{b}\varphi_{1}^{-1},\varphi_{1}\varphi_{2})$$ for all $f_{a},f_{b}\in H$ and $\varphi_{1},\varphi_{2}\in \operatorname{Aut}(G)$.
\begin{proposition}\label{T1}
      Let $G$ be a group.
      \begin{itemize}
          \item [1)]
      For any $m\in \mathbb{Z}$, we have $H\rtimes \operatorname{Aut}(G)\leq \operatorname{Aut(Conj_m}(G))$. 
      \item[2)] If $m=1$ then $ H\rtimes\operatorname{Out}(G) \leq \operatorname{Out(Conj}(G))$.
       \end{itemize}
  \end{proposition}
\begin{proof} 
 \begin{itemize}
     \item [1)]
 
Since $f_{a},\,\, a\in Z(G)$ is an automorphism of $\operatorname{Aut(Conj_{m}}(G))$, we have $$\langle\operatorname{Aut}(G),H \rangle \leq \operatorname{Aut(Conj_m}(G)).$$ Moreover 
 $$
\varphi f_{a}\varphi^{-1}(x)=\varphi(a)x=f_{\varphi(a)}(x),\,x\in G,
 $$
  then $H$ is a normal subgroup  of the group $\langle\operatorname{Aut}(G),H \rangle$. Since $\varphi(1)=1$ for every automorphism $\varphi$ of $G$, the intersection of $\operatorname{Aut}(G)$ and $H$ is trivial. Hence $$H\rtimes \operatorname{Aut}(G)=\langle\operatorname{Aut}(G),H \rangle.$$
\item[2)] We have
   $$\operatorname{Aut(Conj}(G))/Inn(\operatorname{Conj}(G))\leq (H\rtimes\operatorname{Aut}(G) )/Inn(\operatorname{Conj}(G))\leq H\rtimes(\operatorname{Aut}(G) /Inn(G)).$$
    \end{itemize}
\end{proof}
In the following theorem, we formulate a sufficient condition for an antiautomorphism of $G$ to induce an automorphism of $\operatorname{Conj_{m}}(G)$.  We also prove that there are no antiautomorphisms of $\operatorname{Conj_{m}}(G)$.

\begin{theorem}
    Let $G$ be a group and $ m \in \mathbb{Z}$. Then :
        \begin{itemize}
            \item [a)] The intersection $\operatorname{AAut}(G)\cap\Aut(\Conj_m(G))$ is nonempty if and only if $y^{2m}\in Z(G)$ for all $y\in G$.
            \item[b)]   If $G$ is a non-trivial group, then there are no antiautomorphisms  of $\operatorname{Conj_{m}}(G)$ .
        \end{itemize}
\end{theorem}

\begin{proof}
        \begin{itemize}

         \item[a)] 
         It is need to check that $\psi(x\ast y)=\psi(x)\ast \psi(y),\ x,y \in G$, where
$a\ast b=a^{-m}ab^{m},$ $a,b\in G$.
The right hand side is
         \begin{eqnarray*} 
    \psi(x)\ast \psi(y)&=&\psi(y)^{-m}\psi(x)\psi(y)^{m}\\&=& \psi(y^{-m})\psi(x)\psi(y^{m})\\&=&\psi(y^{m}xy^{-m})
\end{eqnarray*}
and the left hand side is $\psi(x\ast y)=\psi(y^{-m}xy^{m})$.
The equality $$\psi(x\ast y)=\psi(x)\ast \psi(y)$$ holds if and only if $y^{-m}xy^{m}=y^{m}xy^{-m}$. Hence
     $x=y^{2m}xy^{-2m}$ i.e. $y^{2m}\in Z(G)$ for any $y\in G$.

\item[b)] Let $\psi$ be an antiautomorphism of   $\operatorname{Conj_{m}}(G)$ and $x\in G$ such that $x\neq 1$.
\\
      We have  
      $$\psi(x)=\psi(x\ast1)=\psi(1)\ast\psi(x)=\psi(x)^{-m}\psi(1)\psi(x)^{m}.
      $$
      Thus $\psi(x)=\psi(1)$ which is a contradiction since $\psi$ is a bijection. 

\end{itemize}
        \end{proof}
\begin{corollary}\label{corollary2.3}
            Let $G$ be a  group. If $G$ is centerless, then there are no antiautomorphisms of $G$ that induce automorphisms of $\Conj(G)$.
        \end{corollary}
        The following are two important examples that illustrate this corollary.
        \begin{example}
         It is well known that the center of the symmetric group $\Sigma_{n},\,\ n\geq 3$ is trivial. Then, by corollary \ref{corollary2.3}, no antiautomorphism of $\Sigma_{n}$ induces an automorphism on $\operatorname{Conj}(\Sigma_{n})$ for $n\geq 3$ . 
        \end{example}
 \begin{example}
     If $S$ is a set containing at least two distinct elements, then the center of the free group $F(S)$ is trivial. In this case, no antiautomorphism of $F(S)$ induces an automorphism on $\operatorname{Conj}(F(S))$.
 \end{example}
The following theorem provides sufficient conditions under which an antiautomorphism of a group $G$ induces an automorphism or antiautomorphism of the Alexander quandle $\operatorname{Alex}(G,\phi)$.  
\begin{theorem}\label{theorem 2.4}
    Let $G$ be a group, $ \operatorname{Aut(Alex }(G, \phi))$  the generalized Alexander quandle and $\psi\in C_{\operatorname{AAut}(G)}(\phi)$. Then
        \begin{itemize}
            \item [a)]
 $\psi\in \operatorname{Aut(Alex }(G, \phi))$ if and only if $\phi$ is a central automorphism,
            \item[b)]  $\psi\in \operatorname{AAut(Alex} (G, \phi))$ if only if  $G$ is an abelian.
        \end{itemize}
       
\end{theorem}
 
\begin{proof}
   
    \begin{itemize}
        \item [a)] We will show that $\psi$ induces an automorphism of $\operatorname{Alex}(G,\phi )$ if and only if $\phi$ is a central automorphism. For all $x,y\in G$,
the left hand side $$\psi(x\ast y)=\psi(\phi(x)\phi(y^{-1})y)=\psi(y)\psi\phi(y^{-1})\psi\phi(x)$$
and the right hand side $$\psi(x)\ast \psi(y)=\phi(\psi(x)\psi(y)^{-1})\psi(y)=\phi\psi(x)\phi\psi(y^{-1})\psi(y).$$
Let us pose $\psi\phi=\phi\psi=f$. The equality $$\psi(x\ast y)=\psi(x)\ast \psi(y)$$ holds if and only if
$$\psi(y)f(y^{-1})f(x)=f(x)f(y^{-1})\psi(y)$$ which implies $$f^{-1}\psi(y)y^{-1}x=xy^{-1}f^{-1}\psi(y)$$ i.e.
\begin{eqnarray*}
1&=&x^{-1}yf^{-1}\psi(y)^{-1}xy^{-1}f^{-1}\psi(y)\\&=&x^{-1}yf^{-1}\psi(y^{-1})xy^{-1}xx^{-1}f^{-1}\psi(y)\\&=&[y^{-1}x,x^{-1}f^{-1}\psi(y)]\\&=&[z,z^{-1}y^{-1}f^{-1}\psi(y)]\,\,\,\text{where}\,\,\, z=y^{-1}x\\&=&[z,y^{-1}\phi^{-1}(y)],
\end{eqnarray*}
hence $y^{-1}\phi^{-1}(y)\in Z(G)$.
\item[b)] We will prove that $\psi$ induces an antiautomorphism of
$\operatorname{Alex}(G,\phi )$ if and only if $G$ is  abelian.
   For all $x,y \in G$, we have : $$\psi(x\ast y)=\psi(\phi(xy^{-1})y)=\psi(y)\psi\phi(xy^{-1})$$ and
   \begin{eqnarray*}
\psi(y)\ast\psi(x)&=&\phi(\psi(y))\phi(\psi(x)^{-1})\psi(x)\\&=&\phi\psi(y)\phi\psi(x^{-1})\psi(x)\\&=&\phi\psi(x^{-1}y)\psi(x).
\end{eqnarray*}
Let us pose $\psi\phi=\phi\psi=f$. The equality
$$\psi(x\ast y)=\psi(y)\ast\psi(x)$$ holds if and only if $$\psi(y)f(xy^{-1})=f(x^{-1}y)\psi(x).$$ Since $f$ is an antiautomorphism, we have $$xy^{-1}f^{-1}\psi(y)=f^{-1}\psi(x)x^{-1}y$$ i.e.
$$f^{-1}\psi(y)=yx^{-1}f^{-1}\psi(x)x^{-1}y.$$ Taking $x=1$ yields $f^{-1}\psi(y)=y^{2}$ . Since $f^{-1}\psi$ is an automorphism, we have $$f^{-1}\psi(a)f^{-1}\psi(b)=f^{-1}\psi(ab),\,\, a,b \in G.$$ This implies that $a^{2}b^{2}=abab$ i.e. $ab=ba$, so $G$ is abelian.
    \end{itemize}
\end{proof}
Analogously the following theorem hold.
\begin{theorem}
    Let $G$ be a group. The intersection $C_{\Aut(G)}(\phi)\cap\operatorname{AAut}(\Alex(G,\phi))$ is nonempty if and only if $G$ is abelian.
\end{theorem}
We conclude this section with the following proposition, which  generalizes the result \cite[Proposition $4$.$1$]{BDS}.

\begin{proposition}\label{proposition}
    Let $G$ a group. Then  $G^{op} \rtimes C_{\operatorname{Aut}(G)}(\phi)\leq \operatorname{Aut(Alex} (G, \phi))$.
\end{proposition}
\begin{proof} Since $F\cong G^{op}$ for $a=1$, we can prove it using similair reasoning as in  \cite[Proposition.$4$.$1$]{BDS}.
\end{proof}
\section{ Automorphisms and antiautomorphisms on   $\operatorname{Core}(G)$}
\label{s3}
 For a group $G$, we determine the
connections between $\operatorname{Aut}(G)$, $\operatorname{AAut}(G)$, $\operatorname{Aut(Core}{(G)})$ and $\operatorname{AAut(Core}{(G)})$.
 Let us now prove
\begin{proposition}\label{Thm4}
    Let $G$ be a group and $\psi\in\operatorname{AAut}(G)$. Then
        \begin{itemize}
            \item [a)]
$\psi \in \operatorname{Aut(Core}{(G)})$,
            \item[b)]   $\psi\in \operatorname{AAut(Core}{(G)})$ if and only if $G$  has exponent $3$.
        \end{itemize}

\end{proposition}

\begin{proof}

        \begin{itemize}
            \item [a)] Let us show that $\psi$ induces an automorphism of $\operatorname{Core}{(G)}$. To prove this, we need to check that $\psi(x\ast y)=\psi(x)\ast \psi(y),\ x,y \in G$. The left hand side is equal to $$\psi(x\ast y)=\psi(y)\psi(x^{-1})\psi(y)$$
and the right hand side is equal to
 $$\psi(x)\ast \psi(y)=\psi(y)\psi(x)^{-1}\psi(y)=\psi(y)\psi(x^{-1})\psi(y).$$ Therfore, we have
$\psi\in \operatorname{Aut(Core}{(G)})$.
\item[b)] For all $x,y\in G$, we have the following :
            \begin{eqnarray*}
    \psi(x\ast y)&=&\psi(yx^{-1}y)\\&=&\psi(y)\psi(x^{-1})\psi(y)\\&=& \psi(y)\psi(x)^{-1}\psi(y)
\end{eqnarray*} and
           $\psi( y)\ast\psi( x)=\psi( x)\psi( y)^{-1}\psi( x).$
  The  equality $$\psi(x\ast y)=\psi( y)\ast\psi( x)$$  holds if and only if $$\psi(y)\psi(x)^{-1}\psi(y) =\psi( x)\psi( y)^{-1}\psi( x),$$ which implies $\psi((x^{-1}y)^{3})=1$ i.e. $(x^{-1}y)^{3}=1$. Hence $\psi$ induces antiautomorphism of $\operatorname{Core}{(G)}$ if only if $G$  has exponent $3$.
\end{itemize}
\end{proof}

Analogously, the following proposition hold.
\begin{proposition}
      Let $G$ be a group and $\varphi\in\operatorname{Aut}(G)$.  Then  $\varphi\in \operatorname{AAut(Core}{(G)})$ if and only if $G$  has exponent $3$.
\end{proposition}
\begin{remark}
    $G$ has exponent $3$ if and only if $\operatorname{Core}(G)$ is commutative (see \cite{LT}). Therefore if $G$ has exponent $3$, then $\operatorname{AAut(Core}{(G)})=\operatorname{Aut(Core}{(G)})$.
\end{remark}
In the following corollary, we determine a condition for the existence of an isomorphism between the group $G$ and the group $\mathbb{Z}_{3}$.
\begin{corollary}
 Let $G$ be the cyclic group of order $n$. Then $\Aut(G)\cap\Aut(\Core(G))$ is nonempty if and only if $n=3$.
    \end{corollary}
    \begin{proof}
      Let $G$ be a cyclic group and $\varphi \in \operatorname{Aut}(G)$.  If $\varphi \in \operatorname{AAut(Core}(G))$ then $G$ has an exponent $3$ by proposition \ref{Thm4}. Thus the smallest common multiple of the order elements of $G$ is $3$. Since $G$ is cyclic, there exists an element $x\in G$ such that $x^{n}=1$, where $n$ is order of $G$. Hence $n$ divides $3$, so $n=3$ i.e., $G \cong \mathbb{Z}_{3}$.

    \end{proof}
\begin{corollary}
    Let $G$ be a group with a trivial center. Then,  no automorphism or antiautomorphism of $G$  induces an antiautomorphism on $Core(G)$.
\end{corollary}
\begin{proof}
    Let $G$ be a group with trivial center and let $\varphi$ be an element of $\operatorname{Aut}(G)$ or $ \operatorname{AAut}(G)$. Suppose that $\varphi \in \operatorname{AAut(Core}(G))$. Then, by proposition \ref{Thm4}, $G$ has an exponent $3$. Therefore $G$ is a $3$-group. Therefore $G$ is nilpotent. Thus $Z(G)$ is non trivial, which is a contradiction. Hence, there is no automorphism or  antiautomorphism of $G$ that induces an antiautomorphism on $\operatorname{Core}(G)$.
\end{proof}
Since the dihedral quandle $\operatorname{R_{3}}$ is commutative, every antiautomorphism of $\operatorname{R_{3}}$ is an automorphism of $\operatorname{R_{3}}$. We  prove that there are no antiautomorphisms on the dihedral quandle $\operatorname{R}_{n}$ of order $n\neq 3$. Recall that the dihedral quandle $\operatorname{R}_{n}$ of order $n$ is a quandle with the operation $a_{i}\ast a_{j}=a_{2j-i\pmod n}$ for $a_{i}, a_{j}\in\operatorname{R}_{n}$. For convenience, we will denote $a_{i\pmod n}$ by $a_{\overline{i}}$ in the following theorem.
\begin{theorem}\label{Theorem 3.4}
     Let $\operatorname{R_{n}}=\{a_{\overline{0}},\cdots a_{\overline{n-1}}\}$ be the dihedral quandle where $n\neq 3$. Then 
    $\operatorname{R_{n}}$ does not have antiautomorphisms.
\end{theorem}

\begin{proof}
    Let $\varphi$ be an antiautomorphism of $\operatorname{R_{n}}$ such that $\varphi(a_{\overline{0}})=a_{\overline {i}}$ and $\varphi(a_{\overline{1}})=a_{\overline{j}}$.
    \\
    Since $\varphi(a_{\overline{0}}\ast a_{\overline{1}})=\varphi(a_{\overline{2}})$ and  $\varphi(a_{\overline{1}})\ast \varphi(a_{\overline{0}})=a_{\overline{j}}\ast a_{\overline{i}}=a_{\overline{2i-j}},$ we have
$\varphi(a_{\overline{2}})=a_{\overline{2i-j}}$.
\\
Likewise
      $$\varphi(a_{\overline{3}})=\varphi(a_{\overline{1}}\ast a_{\overline{2}})=\varphi(a_{\overline{2}})\ast \varphi(a_{\overline{1}})=a_{\overline{2i-j}}\ast a_{j}=a_{\overline{3j-2i}},$$
      
      $$\varphi(a_{\overline{4}})=\varphi(a_{\overline{2}}\ast a_{\overline{3}})=\varphi(a_{\overline{3}})\ast \varphi(a_{\overline{2}})=a_{\overline{3j-2i}}\ast a_{\overline{2i-j}}=a_{\overline{6i-5j}}$$ and
    
     $$\varphi(a_{\overline{5}})=\varphi(a_{\overline{1}}\ast a_{\overline{3}})=\varphi(a_{\overline{3}})\ast \varphi(a_{\overline{1}})=a_{\overline{3j-2i}}\ast a_{\overline{j}}=a_{\overline{2i-j}}.$$
    \\
    Since $a_{\overline{2}}\neq a_{\overline{5}}$  then $\varphi(a_{\overline{2}})=\varphi(a_{\overline{5}})$ is a contradiction. Therefore, there are no antiautomorphisms of
    $\operatorname{R_{n}}$.
\end{proof}
Next, we construct automorphisms of $\operatorname{Core}(G)$ and $\operatorname{(Alex} (G, \phi))$  that are not automorphisms of $G$. To perform  these constructions, we introduce the set of maps $$F=\{f_{a,b}: G\to G\,\,|\,\,a,b\in G,f_{a,b}(x)=axb, x\in G \}.$$  This set is a group under multiplication $f_{a,b}f_{c,d}=f_{ac,db}, $ where $\,a,b,c,d\in G$. We denote  the subgroup $\langle f_{a,b}\,\,|\,\,a\in \operatorname{Fix}(\phi), b\in G\rangle$ of $F$ by $F'$.
\begin{proposition}\label{pro 1}
     Let $G$ a group and $\phi \in \operatorname{Aut}(G)$. Let $N$ be a subset of $G\times G^{op}$ such that $N=\{(a,a^{-1})\,\,|\,\,a\in Z(G)\}\cong Z(G)$.
    \begin{itemize}
        \item [1)] $F$ is a subgroup of $\operatorname{Aut(Core}(G))$.
       \item [2)] $F'$ is a subgroup of $\operatorname{Aut(Alex} (G, \phi))$.
       \item[3)] $F\cong (G\times G^{op})/N$.
    \end{itemize} 
\end{proposition}
\begin{proof} 
1) It is enough to prove that $f_{a,b} \in \operatorname{Aut(Core}(G))$. For all $x,y \in G$, we have 
\begin{eqnarray*}
  f_{a,b}(x)\ast f_{a,b}(y)=(axb)\ast(ayb)&=&aybb^{-1}x^{-1}a^{-1}ayb\\&=&ayx^{-1}yb=f_{a,b}(x\ast y).
\end{eqnarray*}
Hence $F$ is a subgroup of $\operatorname{Aut(Core}(G))$.
\begin{itemize}
\item [2)] We will prove that  $f_{a,b} \in \operatorname{Aut(Alex} (G, \phi))$ where $a\in \operatorname{Fix}(\phi)$. For all $x,y \in G$, we have :
\begin{eqnarray*}
    f_{a,b}(x)\ast f_{a,b}(y)=(axb)\ast(ayb)&=&\phi(axbb^{-1}y^{-1}a^{-1})ayb\\&=&a\phi(xy^{-1})a^{-1}ayb=a\phi(xy^{-1})yb=f_{a,b}(x\ast y)
    \end{eqnarray*}
    Hence $f_{a,b}\in \operatorname{Aut(Alex} (G, \phi))$.
\item[3)] For all $(a,b)\in G\times G^{op}$, we denote by $[(a,b)]$ the equivalence class of $(a,b)$ in $(G\times G^{op})/N$.\\
    Let $\Phi :F \to (G\times G^{op})/N $ such that $\Phi(f_{a,b})=[(a,b)],\, a,b\in G$.\\ For all $a_{1}, b_{1},a_{2},b_{2} \in G$, we have : $$\Phi(f_{a_{1},b_{1}}f_{a_{2},b_{2}})=\Phi(f_{a_{1}a_{2},b_{2}b_{1}})=[(a_{1}a_{2},b_{2}b_{1})]$$  and 
    $$\Phi(f_{a_{1},b_{1}})\Phi(f_{a_{2},b_{2}})=[(a_{1},b_{1})][(a_{2},b_{2})]=[(a_{1}a_{2},b_{2}b_{1})].$$ If $f_{a,b}\in \operatorname{Ker}(\Phi)$, then $\Phi(f_{a,b})=[(a,b)]=[(1,1)]$. Thus there exists $c\in Z(G)$ such that $(a,b)=(c,c^{-1})$ i.e. $b=a^{-1},\,a\in Z(G).$ Hence $f_{a,b}(x)=f_{a,a^{-1}}(x)=x$ for all $x\in G$ and $\operatorname{Ker}(\Phi)=\{\operatorname{id}\}$, where $\operatorname{id}$ is the identity map. Additionally, for all $[(a,b)]\in (G\times G^{op})/N$, we have $a,b \in G$ and $f_{a,b}\in F$. Thus $\Phi(f_{a,b})=[(a,b)]$ and $[(a,b)]\in \operatorname{Im}\Phi$. Hence $\operatorname{Im}\Phi=(G\times G^{op})/N.$ This completes the proof.
\end{itemize}
\end{proof}
\begin{proposition}\label{proposition 2.10}
    Let $G$ be a group and let $N$ be a set such that \\$N=\{(a,a^{-1})\,\,|\,\,a\in Z(G)\}\cong Z(G)$. Then $((G\times G^{op})/N)\rtimes\operatorname{Out}(G)\leq \operatorname{Aut(Core} (G)).$
\end{proposition}
\begin{proof}
    It is enough to prove that $F\rtimes\operatorname{Out}(G) \leq \operatorname{Aut(Core} (G)).$ We have $\langle \operatorname{Out}(G), F\rangle \leq \operatorname{Aut(Core} (G))$ by proposition \ref{pro 1}.  Let $a,b\in G$ and  $\varphi \in \operatorname{Out}(G)$. For $x\in G$, we have  $$\varphi^{-1}f_{a,b}\varphi(x)=\varphi^{-1}(a)x\varphi^{-1}(b)=f_{\varphi^{-1}(a),\varphi^{-1}(b)}(x).$$ Thus, $F$ is the normal subgroup  of the group $\langle\operatorname{Out}(G), F\rangle$. Since every automorphism $\varphi$ of $G$ satisfies $\varphi(1)=1$, the intersection of $\operatorname{Out}(G)$ and $F$ is  the group of inner automorphisms of $G$. Hence $F\rtimes\operatorname{Out}(G)\leq \operatorname{Aut(Core} (G)).$
\end{proof}
We present a theorem similar to \cite[Theorem.$4$.$2$]{BDS} for an abelian group without $2$-torsion.
\begin{theorem}
    Let $G$ be an abelian group without $2$-torsion. Then $\operatorname{Aut(Core} (G))\cong G \rtimes \operatorname{Aut}(G)$.
\end{theorem}
\begin{proof}
    Let $e$ be the neutral element of $G$ and  consider the map
    $\Phi:  G \rtimes \operatorname{Aut}(G) \to \operatorname{Aut(Core }(G))$ defined by $\Phi((a,h))=t_{a}h$ where $t_{a}:G \to G$ given by $t_{a}(b)=ba$ for all $b\in G$ . By the proof of \cite[Proposition $4$. $1$]{BDS}, $\Phi$ is an homomorphism.
\\
   We prove that $\Phi$ is injective. $$\Phi(a_{1},h_{1})=\Phi(a_{2},h_{2})\,\,\text{i.e.}\,\,t_{a_{1}}h_{1}=t_{a_{2}}h_{2}.$$ Thus $t_{a_{1}}h_{1}(e)=t_{a_{2}}h_{2}(e)$ which equivalent to $a_{1}e=a_{2}e$ i.e. $a_{1}=a_{2}$.
 Hence  $h_{1}=h_{2}$ i.e. $(a_{1},h_{1})=(a_{2},h_{2})$.
\\
 Now, we prove that $\Phi$ is surjective. Let $f\in \operatorname{Aut(Core }(G))$. Then for all $x,y \in G$, we have $$f(x\ast y)=f(x)\ast f(y)\,\,\text{i.e.}\,\, f(xy^{-1}x)=f(x)f(y)^{-1}f(x).$$ Define $h: G \to G$ by $h(x)=f(x)f(e)^{-1}$. Since $f$ is a bijection, to show that $h\in \operatorname{Aut}(G)$, it is enough to show that $f(xy)=f(x)f(y)f(e)^{-1}$.\\
The left hand side
    \begin{eqnarray*}
       f(xy)&=&f(x(xy^{-1})^{-1}x)\\&=&f(x)f(xy^{-1})^{-1}f(x)=f(x)^{2}f(xy^{-1})^{-1}. \,\,\,\, (1)
    \end{eqnarray*}
    On the other hand side we have
  $f(xy)=f(yx)=f(y)^{2}f(yx^{-1})^{-1}. \,\,\,\, (2)$\\
  Since
    \begin{eqnarray*}
f(xy^{-1})&=&f(e(yx^{-1})^{-1}e)\\&=&f(e)f(yx^{-1})^{-1}f(e)=f(e)^{2}f(yx^{-1})^{-1}.
    \end{eqnarray*}
   Multiplying $(1)$ and $(2)$ and replacing  $f(xy^{-1})$ yields
$f(xy)^{2}=f(x)^2f(y)^2f(e)^{-2}$, i.e. $$[f(xy)(f(x)f(y)f(e)^{-1})^{-1}]^{2}=e.$$Since $G$ has no $2$-torsion, we have $$f(xy)(f(x)f(y)f(e)^{-1})^{-1}=e ,\,\,\text{ i.e.}\,\, f(xy)=f(x)f(y)f(e)^{-1}.$$ Hence $h\in \operatorname{Aut}(G)$. Since $h=f(e)^{-1}f$ is equivalent to
     $f=f(e)h\in \Phi( G \rtimes \operatorname{Aut}(G)),\,\,\Phi$ is surjective.
\end{proof}

%%%%%%%%%%%%%%%%%%%%%%%%%%%%%%%%%%%%%%%%%%%%%%%%%%%%%%%%%%%%%%%%%%%%%%%%%%%%%%%%%%%%%%%%%%%%%%%%%%%%%%%%%%%%%%%%%

\section{Automorphisms and antiautomorphisms of some analogous of  $\Alex(G,\phi)$} \label{s4}

For the quandles defined in section \ref{sec-prelim}, we  study the following question: what are the connections between  $\operatorname{Aut}(G)$, $\operatorname{AAut}(G)$, $\operatorname{Aut}(Q_{i})$ and $\operatorname{AAut}(Q_{i})$?

\begin{theorem}\label{theorem 6} Let $G$ be a group and $\varphi\in C_{\operatorname{Aut}(G)}(\phi)$. Then
\begin{itemize}
    \item [1)] $C_{\operatorname{Aut}(G)}(\phi)\leq\operatorname{Aut}(Q_{i})$ for $i=1,2,3,4$.
    \item [2)] $\varphi\in \operatorname{AAut}(Q_{i})$ if and only if $G$ is abelian, for $i=1,2$,
    \item [3)] $\varphi\in \operatorname{AAut}(Q_{i})$ if and only if $[G,G]$ is a central subgroup with exponent $3$ of $G$, for $i=3,4$.
    \end{itemize}
\end{theorem}
\begin{proof}  
 \begin{itemize}
        \item [1)] We  prove that $\varphi \in\operatorname{Aut}(Q_{i}(G))$ for $i=1,2,3,4$.
    \end{itemize}
    \begin{itemize}
       \item [-] Let $i=1$,
        for all  $ x,y \in G$, we have $\varphi(x\ast_{1} y)=\varphi(y\phi(y^{-1}x))=\varphi(y)\varphi\phi(y^{-1}x)$ and
        \begin{eqnarray*}
            \varphi(x)\ast_{1} \varphi(y)=\varphi(y)\phi(\varphi(y)^{-1}\varphi(x))=\varphi(y)\phi(\varphi(y^{-1})\varphi(x))&=&\varphi(y)\phi(\varphi(y^{-1}x))\\&=&\varphi(y)\phi\varphi(y^{-1}x)\\&=& \varphi(x\ast_{1} y).
        \end{eqnarray*}
        Hence $\varphi \in \operatorname{Aut}(Q_{1}(G))$ i.e. $C_{\operatorname{Aut}(G)}(\phi)\leq \operatorname{Aut}(Q_{1}(G))$.
        \item [-] Let $i=2$,
for all $x,y \in G$, we have $\varphi(x\ast_{2} y)=\varphi(y\phi(yx^{-1}))=\varphi(y)\varphi\phi(yx^{-1})$ and

           $$ \varphi(x)\ast_{2}  \varphi(y)=\varphi(y)\phi(\varphi(y)\varphi(x)^{-1})=\varphi(y)\phi\varphi(yx^{-1}).$$

        Hence $\varphi \in \operatorname{Aut}(Q_{2}(G))$ i.e. $C_{\operatorname{Aut}(G)}(\phi)\leq \operatorname{Aut}(Q_{2}(G))$.
        \item [-] Let  $i=3$, for all  $x,y \in G$, we have $\varphi(x\ast_{3} y)=\varphi(\phi(xy^{-1})y)=\varphi\phi(xy^{-1})\varphi(y)$ and

            $$ \varphi(x)\ast_{3} \varphi(y)=\phi(\varphi(x)\varphi(y)^{-1})\varphi(y)=\phi\varphi(xy^{-1})\varphi(y).$$
        Hence $\varphi \in \operatorname{Aut}(Q_{3}(G))$ i.e. $C_{\operatorname{Aut}(G)}(\phi)\leq \operatorname{Aut}(Q_{3}(G))$.
        \item [-] Let $i=4$,
for all  $x,y \in G$, we have $\varphi(x\ast_{4} y)=\varphi(y\phi(y^{-1}x))=\varphi(y)\varphi\phi(y^{-1}x)$ and

           $$ \varphi(x)\ast_{4}  \varphi(y)=\varphi(y)\phi(\varphi(y)^{-1}\varphi(x))=\varphi(y)\phi\varphi(y^{-1}x).$$

        Hence $\varphi \in \operatorname{Aut}(Q_{4}(G))$ i.e. $C_{\operatorname{Aut}(G)}(\phi)\leq \operatorname{Aut}(Q_{4}(G))$.
    \end{itemize}
  Throughout the
 proof of $2)$ and $3)$, suppose that $f=\varphi\phi=\phi\varphi$.
    \begin{itemize}
     \item[2)] We  prove that $\varphi \in\operatorname{AAut}(Q_{i}(G))$  if and only if $G$ is abelian, for $i=1,2$.
     \end{itemize}
      \begin{itemize}
       \item [-] Let $i=1$,
    for all $x,y \in G$, we have     $\varphi(x\ast_{1} y)=\varphi(y\phi(y^{-1}x))=\varphi(y)\varphi\phi(y^{-1}x)$ and
        \begin{eqnarray*}
            \varphi(y)\ast_{1} \varphi(x)=\varphi(x)\phi(\varphi(x)^{-1}\varphi(y))&=&\varphi(x)\phi(\varphi(x^{-1}y))\\&=&\varphi(x)\phi\varphi(x^{-1}y).
        \end{eqnarray*}
The equality  $$\varphi(x\ast_{1} y)=\varphi(y)\ast_{1}  \varphi(x)$$ holds if and only if $$\varphi(y)\varphi\phi(y^{-1}x)=\varphi(x)\phi\varphi(x^{-1}y)$$ i.e.
\\
$\varphi(y)f(y^{-1}x)=\varphi(x)f(x^{-1}y)$. Acting of $f^{-1}$ on both sides, we obtain $$f^{-1}\varphi(y)y^{-1}x= f^{-1}\varphi(x)x^{-1}y \,\,\text{i.e.}\,\, f^{-1}\varphi(x^{-1}y)=(x^{-1}y)^{2}.$$ If $x = 1$, then $f^{-1}\varphi(y)=y^{2}$. Therefore, $G$  is abelian.
        \item [-] Let $i=2$, for
for all  $x,y \in G$, we have $\varphi(x\ast_{2} y)=\varphi(y\phi(yx^{-1}))=\varphi(y)\varphi\phi(yx^{-1})$ and
        \begin{eqnarray*}
            \varphi(y)\ast_{2}  \varphi(x)=\varphi(x)\phi(\varphi(x)\varphi(y)^{-1})&=&\varphi(x)\phi(\varphi(xy^{-1}))\\&=&\varphi(x)\phi\varphi(xy^{-1}).
        \end{eqnarray*}
The equality $$\varphi(x\ast_{2} y)=\varphi(y)\ast_{2} \varphi(x)$$ holds if and only if $$\varphi(y)\varphi\phi(yx^{-1})=\varphi(x)\phi\varphi(xy^{-1}).$$
     By  action of $f^{-1}$ on both sides, we obtain
$$f^{-1}\varphi(x^{-1}y)=(xy^{-1})^{2}.$$ Taking $x = 1$ yields $f^{-1}\varphi(y)=(y^{-1})^{2}$. Since $f^{-1}\varphi$ is a antiautomorphism, we have for all $a,b\in G$ : $$(ab)^{-1}(ab)^{-1}=b^{-1}b^{-1}a^{-1}a^{-1}$$ i.e. $$b^{-1}a^{-1}b^{-1}a^{-1}=b^{-1}b^{-1}a^{-1}a^{-1}.$$ Hence $a^{-1}b^{-1}=b^{-1}a^{-1}$ i.e. $G$ is abelian.
        \end{itemize}
        \begin{itemize}
     \item[3)] We  prove that $\varphi \in\operatorname{AAut}(Q_{i})$  if and only if $[G,G]$ is a central subgroup with exponent $3$ of $G$, for $i=3,4$.
     \end{itemize}
     \begin{itemize}
         \item [-] Let $i=3$, 
        for all $x,y \in G,$ we have $\varphi(x\ast_{3} y)=\varphi(\phi(xy^{-1})y)=\varphi\phi(xy^{-1})\varphi(y)$ and
        $$ \varphi(y)\ast_{3} \varphi(x)=\phi(\varphi(y)\varphi(x)^{-1})\varphi(x)=\phi\varphi(yx^{-1})\varphi(x).$$
The equality  $$\varphi(x\ast_{3} y)= \varphi(y)\ast_{3} \varphi(x)$$ holds if and only if $$ f^{-1}\varphi(yx^{-1})=(yx^{-1})^{2},$$ where $f=\varphi\phi=\phi\varphi$. Since $f^{-1}\varphi$ is a antiautomorphism, we have $(ab)^{2}=b^{2}a^{2},\,\, a,b \in G$ i.e. $abab=b^{2}a^{2}$. Thus by action of $b^{-1}$ on the left and $a^{-1}$ on the left , we obtain $$a^{-1}b^{-1}abab=a^{-1}ba^{2}$$ i.e. $$[a,b]ab=b(b^{-1}a^{-1}ba)a=b[b,a]a.$$ Since $[b,a]=[a,b]^{-1}$, we have $[a,b]^{2}ab=ba$ which equivalent to $[a,b]^{2}=b^{-1}a^{-1}ba$ i.e. $[a,b]^{3}=1$.
\\
The proof for $Q_4$ is similar.
     \end{itemize}
\end{proof}
Analogously, the following theorem hold.
\begin{theorem}
Let $G$ be a group and $\psi\in C_{\operatorname{AAut}(G)}(\phi)$. Then  
\begin{itemize}
    \item [1)] $\psi\in \operatorname{AAut}(Q_{i})$ if and only if $G$ is abelian, for $i=1,2$,
    \item [2)] $\psi\in \operatorname{AAut}(Q_{i})$ if and only if $[G,G]$ is a central subgroup with exponent $3$ of $G$, for $i=3,4$.
    \end{itemize}
\end{theorem}
Next, we provide the conditions under which an autiautomorphism of group $G$ induces an automorphism of $Q_{i},\,i=1,2,3,4$.
\begin{theorem}\label{theorem 3.2} Let $G$ be a group  and $\psi\in C_{\operatorname{AAut}(G)}(\phi)$. Then
\begin{itemize}
    \item [1)]

 $\psi\in \operatorname{Aut}(Q_{i})$ if and only if $\phi$ is a central automorphism, for $i=1,2$.
 \item[2)] $\psi\in \operatorname{Aut}(Q_{i})$ if and only if $y\phi^{-1}(y)\in Z(G)$, for $i=3,4$.
\end{itemize}
    \begin{proof} Throughout the
 proof, suppose that $\psi\phi=\phi\psi=f.$
    \begin{itemize}
    \item [1)]
    
           We  prove that $\psi \in \operatorname{Aut}(Q_{i})$ if and only if $\phi$ is a central automorphism, for $i=1,2$.
        \begin{itemize}
            \item [-] Let $i=1$,
for all $x,y \in G$, we have $\psi(x\ast_{1}  y)=\psi(y\phi(y^{-1}x))=\psi\phi(y^{-1}x)\psi(y)$ and
        \begin{eqnarray*}
         \psi(x)\ast_{1}  \psi(y)&=&\psi(y)\phi(\psi(y)^{-1}\psi(x))\\&=&\psi(y)\phi(\psi(y^{-1})\psi(x))=\psi(y)\phi\psi(xy^{-1}).
         \end{eqnarray*}
          The equality $$\psi(x\ast_{1}  y)=\psi(x)\ast_{1}  \psi(y)$$ holds if and only if
$$f^{-1}\psi(y)y^{-1}x=xy^{-1}f^{-1}\psi(y)$$ i.e.
         \begin{eqnarray*}
             1&=&x^{-1}yf^{-1}\psi(y^{-1})xy^{-1}f^{-1}\psi(y)\\&=&[y^{-1}x,x^{-1}f^{-1}\psi(y)]\\&=&[z,z^{-1}y^{-1}f^{-1}\psi(y)]\,\,\,\text{where}\,\,\, z=y^{-1}x\\&=&[z,z^{-1}y^{-1}\phi^{-1}(y)]\\&=&[z,y^{-1}\phi^{-1}(y)],
         \end{eqnarray*}
         hence $y^{-1}\phi^{-1}(y)\in Z(G)$.
 \item [-] Let $i=2$,
for all $x,y \in G$, we have $\psi(x\ast_{2}  y)=\psi(y\phi(yx^{-1}))=\psi\phi(yx^{-1})\psi(y)$ and
\begin{eqnarray*}
    \psi(x)\ast_{2} \psi(y)=\psi(y)\phi(\psi(y)\psi(x^{-1}))=\psi(y)\phi\psi(x^{-1}y). 
\end{eqnarray*}
   The equality $$\psi(x\ast_{2} y)=\psi(x)\ast_{2}\psi(y)$$ holds if and only if $f^{-1}\psi(y)x^{-1}y=yx^{-1}f^{-1}\psi(y)$ i.e.
       \begin{eqnarray*}
        1&=&y^{-1}xf^{-1}\psi(y^{-1})yx^{-1}f^{-1}\psi(y)\\&=&[x^{-1}y,y^{-1}f^{-1}\psi(y)]\\&=&[z,y^{-1}\phi^{-1}(y)]\,\,\,\text{where}\,\,\, z=x^{-1}y,
       \end{eqnarray*}
          hence $y^{-1}\phi^{-1}(y)\in Z(G)$.
    \end{itemize}
          \item[2)] We prove that $\psi\in\operatorname{Aut}(Q_{i})$ if and only if $y\phi^{-1}(y)\in Z(G)$, for $i=3,4$.
          
 Let $i=3$,
for all $x,y \in G$, we have $\psi(x\ast_{3}  y)=\psi(\phi(xy^{-1})y)=\psi(y)\psi\phi(xy^{-1})$ and
          \begin{eqnarray*}
              \psi(x)\ast_{3}\psi(y)=\phi(\psi(x)\psi(y^{-1}))\psi(y)=\phi\psi(y^{-1}x))\psi(y).
          \end{eqnarray*}
          The equality $$\psi(x\ast_{3}y)=\psi(x)\ast_{3} \psi(y)$$ holds if and only if $$f^{-1}\psi(y)xy^{-1}=y^{-1}xf^{-1}\psi(y)$$ i.e.
          \begin{eqnarray*}
              1&=&yx^{-1}f^{-1}\psi(y^{-1})y^{-1}xf^{-1}\psi(y)\\&=&yx^{-1}f^{-1}\psi(y^{-1})y^{-1}xy^{-1}yf^{-1}\psi(y)\\&=&[xy^{-1},yf^{-1}\psi(y)]\\&=&[z,y\phi^{-1}(y)]\,\,\,\text{where}\,\,\, z=xy^{-1},
          \end{eqnarray*}
          hence $y\phi^{-1}(y)\in Z(G)$. The proof for $Q_4$ is similar.
        \end{itemize}
    \end{proof}
\end{theorem}

%%%%%%%%%%%%%%%%%%%%%%%%%%%%%%%%%%%%%%%%%%%%%%%%%%%%%%%%%%%%%%%%%%%%%%%%%%%%%%%%%%%%%%%%%%%%%%%%%%%%%%%%%%%%%%%%%

\section{Automorphisms and antiautomorphisms on verbal quandles with one parameter } \label{s5}

 In this  section, we  study the following question: what are the connections between  $\operatorname{Aut}(G)$, $\operatorname{AAut}(G)$, $\operatorname{Aut}(P_{i})$ and $\operatorname{AAut}(P_{i})$?
 
 These quandles $P_{i}$ were defined in  section \ref{sec-prelim}.
\begin{theorem}\label{theorem 4.1}
    Let $G$ be a group, $\varphi\in \operatorname{Aut}(G)$ and $P_{i}(G)$ a verbal quandle with  parameter $c$. Then
 $\varphi$ induces an automorphism of $P_{i}(G)$ if and only if
 $c^{-1}\varphi^{-1}(c)\in Z(G)$.
\end{theorem}
\begin{proof}
    We claim that all the verbal quandles defined above satisfy this result. Therefore, it is suffices to prove it for an arbitrary verbal quandle.
    \\
Let $\varphi\in \operatorname{Aut}(G)$ and $P_{1}(G)$ a verbal quandle with one parameter  defined by
$$x\circ_{1} y=yc^{-1}y^{-1}xc,\,\,x,y \in G.$$ We will proved that $\varphi$ induces an automorphism on $\operatorname{P_{i}}(G)$ if and only if $c^{-1}\varphi^{-1}(c)\in Z(G)$.
$$\varphi(x\circ_{1}y)=\varphi(yc^{-1}y^{-1}xc)=\varphi(y)\varphi(c^{-1}y^{-1}xc)$$ and
$$\varphi(x)\circ_{1} \varphi(y)=\varphi(y)c^{-1}\varphi(y^{-1})\varphi(x)c=\varphi(y)\varphi(\varphi^{-1}(c^{-1})y^{-1}x\varphi^{-1}(c)).$$
The equality $$\varphi(x\circ_{1} y)=\varphi(x)\circ_{1}\varphi(y)$$ holds if only if $$c^{-1}y^{-1}xc=\varphi^{-1}(c^{-1})y^{-1}x\varphi^{-1}(c)$$ i.e. $$y^{-1}x=\varphi^{-1}(c)c^{-1}y^{-1}xc\varphi^{-1}(c^{-1}).$$ Hence $\varphi$ induces automorphism on $\operatorname{P_{1}}(G)$ if and only if $$y^{-1}x=(c\varphi^{-1}(c^{-1}))^{-1}y^{-1}xc\varphi^{-1}(c^{-1})\,\,\text{i.e.}\,\, c\varphi^{-1}(c^{-1})\in Z(G).$$
\end{proof}
\begin{theorem}\label{theorem 4.2}
    Let $G$ be a group,   $c\in G$ and $i\in\{1,2,3,4\}$. 
    \begin{itemize}
        \item [1)] If $\varphi\in \operatorname{Aut}(G)$ and $c\in \operatorname{Fix}(\varphi)$, then  
   $\varphi$ induces an antiautomorphism of $P_{i}$  if and only if
    $x=[x,c^{-1}]$ for all $x\in G$.
    \item[2)]  If $\psi\in \operatorname{AAut}(G)$ and $c\in \operatorname{Fix}(\psi)$, then $\psi$ induces an automorphism on $P_{i}$ if and only if
        $[c^{2},x^{-1}]=1$ for all $ x\in G$.
\end{itemize}
\end{theorem}
\begin{proof}

\begin{itemize}
    \item [1)]
     We prove that  $\varphi$ induces antiautomorphism on $P_{i}$  if and only if
        $x=[x,c^{-1}],\,\, x\in G,\,\, i=1,2,3,4$.\\
            - Let  $i=1$, for $x,y\in G$, we have 
   $$\varphi(x\circ_{1} y)=\varphi(yc^{-1}y^{-1}xc)$$ and
   $$\varphi(y)\circ_{1} \varphi(x)=\varphi(x)c^{-1}\varphi(x^{-1})\varphi(y)c=\varphi(xc^{-1}x^{-1}yc).$$
The equality $$\varphi(x\circ_{1} y)=\varphi(y)\circ_{1} \varphi(x)$$ holds if and only if $$x^{-1}c^{-1}y^{-1}xc=c^{-1}x^{-1}yc.$$
   Thus $$cx^{-1}yc^{-1}y^{-1}x=x^{-1}y.$$
Taking $y = 1$ yields
  $$x=x^{-1}cxc^{-1} \,\,\text{i.e.}\,\, x=[x,c^{-1}],\,\, x\in G.$$
             - Let  $i=2$, for $x,y\in G$, we have
        
         $$\varphi(x\circ_{2}y)=\varphi(yc^{-1}xy^{-1}c)$$ and
         $$\varphi(y)\circ_{2} \varphi(x)=\varphi(x)c^{-1}\varphi(y)\varphi(x^{-1})c=\varphi(xc^{-1}yx^{-1}c).$$
       The equality $$\varphi(x\circ_{2} y)=\varphi(y)\circ_{2} \varphi(x)$$ holds if and only if $$yc^{-1}xy^{-1}=xc^{-1}yx^{-1}\,\,\text{i.e.}\,\,xy^{-1}=x^{-1}ycy^{-1}xc^{-1}.$$
       Taking $y= 1$ yields $x=x^{-1}cxc^{-1}.$\\
         - Let  $i=3$, for $x,y\in G$, we have
        
         $$\varphi(x\circ_{3}y)=\varphi(c^{-1}y^{-1}xcy)$$ and
         $$\varphi(y)\circ_{3} \varphi(x)=c^{-1}\varphi(x)^{-1}\varphi(y)c\varphi(x)=\varphi(c^{-1}x^{-1}ycx).$$
       The equality $$\varphi(x\circ_{} y)=\varphi(y)\circ_{3} \varphi(x)$$ holds if and only if $$c^{-1}y^{-1}xcy=c^{-1}x^{-1}ycx\,\,\text{i.e.}\,\,y^{-1}x=x^{-1}ycxy^{-1}c^{-1}.$$
       Taking $y= 1$ yields $x=x^{-1}cxc^{-1}.$\\
  - Let  $i=4$, for $x,y\in G$, we have
   $$\varphi(x\circ_{4} y)=\varphi(c^{-1}xy^{-1}cy)$$ and
   $$\varphi(y)\circ_{4} \varphi(x)=c^{-1}\varphi(y)\varphi(x^{-1})c\varphi(x)=\varphi(c^{-1}yx^{-1}cx).$$
    The equality $$\varphi(x\circ_{4} y)=\varphi(y)\circ_{4} \varphi(x)$$ holds if and only if $$c^{-1}xy^{-1}cy=c^{-1}yx^{-1}cx.$$
    Hence $xy^{-1}c=yx^{-1}cxy^{-1}$ i.e. $xy^{-1}=[xy^{-1},c^{-1}]$.
\item[2)] We prove that $\psi$ induces an automorphism on $P_{i}$ if and only if
        $[c^{2},x^{-1}]=1,\,\, x\in G$.\\
     - Let  $i=1$, for $x,y\in G$, we have  $$\psi(x\circ_{1} y)=\psi(yc^{-1}y^{-1}xc)$$ and 
$$\psi(x)\circ_{1}\psi(y)=\psi(y)c^{-1}\psi(y)^{-1}\psi(x)c=\psi(cxy^{-1}c^{-1}y).$$\\
The equality  $$\psi(x\circ_{1} y)=\psi(x)\circ_{1} \psi(y)$$ holds if and only if $$yc^{-1}y^{-1}xc=cxy^{-1}c^{-1}y.$$
Taking $y=1$ yields $c^{-2}xc^{2}x^{-1}=1$  i.e.  $[c^{2},x^{-1}]=1$.\\
     - Let  $i=2$, for $x,y\in G$, we have $$\psi(x\circ_{2} y)=\psi(yc^{-1}xy^{-1}c)$$ and
         $$\psi(x)\circ_{2}\psi(y)=\psi(y)c^{-1}\psi(x)\psi(y^{-1})c=\psi(cy^{-1}xc^{-1}y).$$
        The equality  $$\psi(x\circ_{2} y)=\psi(x)\circ_{2} \psi(y)$$ holds if and only if $$yc^{-1}xy^{-1}c=cy^{-1}xc^{-1}y.$$
          Taking $y=1$ yields $c^{-2}xc^{2}x^{-1}=1$  i.e.  $[c^{2},x^{-1}]=1$.\\   
          - Let  $i=3$, for $x,y\in G$, we have $$\psi(x\circ_{3} y)=\psi(c^{-1}y^{-1}xcy)$$ and
         $$\psi(x)\circ_{3}\psi(y)=c^{-1}\psi(y)^{-1}\psi(x)c\psi(y)=\psi(ycxy^{-1}c^{-1}).$$
        The equality  $$\psi(x\circ_{3} y)=\psi(x)\circ_{3} \psi(y)$$ holds if and only if $$c^{-1}y^{-1}xcy=ycxy^{-1}c^{-1}.$$
          Taking $y=1$ yields $c^{-2}xc^{2}x^{-1}=1$  i.e.  $[c^{2},x^{-1}]=1$.\\ 
          - Let  $i=4$, for $x,y\in G$, we have $$\psi(x\circ_{4} y)=\psi(c^{-1}xy^{-1}cy)$$ and
         $$\psi(x)\circ_{4}\psi(y)=c^{-1}\psi(x)\psi(y)^{-1}c\psi(y)=\psi(ycy^{-1}xc^{-1}).$$
        The equality  $$\psi(x\circ_{4} y)=\psi(x)\circ_{4} \psi(y)$$ holds if and only if $$c^{-1}xy^{-1}cy=ycy^{-1}xc^{-1}.$$
          Taking $y=1$ yields $c^{-2}xc^{2}x^{-1}=1$  i.e.  $[c^{2},x^{-1}]=1$.
          \end{itemize}
\end{proof}

%%%%%%%%%%%%%%%%%%%%%%%%%%%%%%%%%%%%%%%%

{\bf Acknowledgements.} I am very grateful to Professor V. G. Bardakov of Novosibirsk State University, the director of this work, for his availability, patience, and attention while carrying out this work. I would also like to thank M. V. Neshchadim, P. Sokolov and Luc Ta  for their helpful suggestions. I would also like to thank the participants of the Evariste Galois seminar.

%%%%%%%%%%%%%%%%%%%%%%%%%%%%%%%%%%%%%%%%%%%%%%%%%%
\begin{center}
  
\end{center}

\end{document}